\documentclass[letterpaper]{amsart}
\usepackage{fullpage}
\usepackage{mathtools}
\usepackage{amsthm,amsmath,amssymb,amscd}
\usepackage{booktabs}
\usepackage[hyperfootnotes=false, colorlinks, linkcolor={blue},
citecolor={magenta}, filecolor={blue}, urlcolor={blue}]{hyperref}
\usepackage{resizegather}
\usepackage{lineno}
\usepackage{adjustbox}

\allowdisplaybreaks
\newcommand{\Q}{{\mathbb Q}}

\newcommand{\Z}{{\mathbb Z}}

\newcommand{\C}{{\mathbb C}}

\newcommand{\p}{\mathfrak p}
\newcommand{\OF}{{\mathfrak o}}
\newcommand{\GL}{{\rm GL}}

\newcommand{\GSp}{{\rm GSp}}

\newcommand{\tr}{{\rm tr}}
\newcommand{\nm}{{\rm n}}

\newcommand{\Mat}{{\rm M}}

\newcommand{\SSp}{{\rm Sp}}

\newcommand*{\transp}[2][-3mu]{\ensuremath{\mskip1mu\prescript{\smash{\mathrm t\mkern#1}}{}{\mathstrut#2}}}

\newtheorem{lemma}{Lemma}[section]
\newtheorem{theorem}[lemma]{Theorem}
\newtheorem{corollary}[lemma]{Corollary}
\newtheorem{proposition}[lemma]{Proposition}

\newtheoremstyle{named}{}{}{\itshape}{}{\bfseries}{.}{.5em}{\thmnote{#3}}
\theoremstyle{named}

\begin{document}
\date{\today}
\title{Further results on the paramodular Hecke algebra}
\author{Jennifer Johnson-Leung}
\address{Department of Mathematics and Statistical Science\\University of Idaho\\Moscow, ID USA}
\email{jenfns@uidaho.edu}
\author{Joshua Parker}
\address{Division of Mathematics and Natural Sciences\\Elmira College\\Elmira, NY USA}
\email{jparker@elmira.edu}
\author{Brooks Roberts}
\address{Department of Mathematics and Statistical Science\\University of Idaho\\Moscow, ID USA}
\email{brooksr@uidaho.edu}

\begin{abstract} In this paper we further examine the structure of the graded paramodular Hecke algebra~$\mathcal H$. We define a commutative subring~$\mathcal{H}^\alpha$ of~$\mathcal{H}$ as the elements fixed by a certain automorphism~$\alpha$, and we show that the elements~$T(q^k)$,
defined as the sum of the double cosets of degree~$k$, are contained in~$\mathcal{H}^\alpha$.
We prove that the formal power series in the indeterminate~$t$ with coefficients~$T(q^k)$ is an
explicit rational function in~$t$ with coefficients from~$\mathcal{H}^\alpha$.
We also determine the structure of~$\mathcal{H}^\alpha$ and its relation to~$\mathcal{H}$.

\end{abstract}
\maketitle

\section{Introduction}\label{sec1}

In this article we build upon our previous work~\cite{JLPR}
on the paramodular Hecke algebra ~$\mathcal{H}$ over a nonarchimedean local field~$F$ of characteristic zero with ring of integers~$\OF$, maximal ideal~$\p = (\varpi)  \subset \OF$,
and residual degree~$q$.
As in~\cite{JLPR},~$\mathcal{H}$ is the graded $\Z$-algebra spanned by the
double cosets in~$K\backslash\Delta/K$,
where~$K$ is the paramodular group of level~$\p$
and~$\Delta$ is a certain semigroup in~$\GSp(4,F)$.
The work~\cite{JLPR} described~$\mathcal{H}$
in terms of four generators and seven relations and also determined
the center of~$\mathcal{H}$.

The first main result of this current work solves a problem that was partially addressed in~\cite{JLPR}.
For each non-negative integer~$k$, let~$T(q^k)$ be the sum of all the double cosets
in~$\mathcal{H}$ of degree~$k$.
In applications to Siegel modular forms, the elements~$T(q^k)$ define natural operators,
and in analogy to results from other settings, one is led to consider
the formal power series
\begin{equation}
\label{formalpowerserieseq}
\sum_{k=0}^\infty T(q^k)t^k,
\end{equation}
in the indeterminate~$t$ and ask, is this power series a rational function in~$t$?
Since~$\mathcal{H}$ is non-commutative this question is somewhat ill-posed.
In~\cite{JLPR} we addressed this question and issue by instead considering
the formal power series~$\sum_{k=0}^\infty \chi(T(q^k)) t^k$ where~$\chi$
is a homomorphism from~$\mathcal{H}$ to~$\C$, and proving that such power
series are rational.
Our current work takes a different approach.
We define a certain maximal commutative~$\Z$-algebra~$\mathcal{H}^\alpha$
of~$\mathcal{H}$ as the fixed points of a certain automorphism~$\alpha$ of~$\mathcal{H}$,
and we prove that~$\mathcal{H}^\alpha$ is commutative;
we also prove that each of the elements~$T(q^k)$ is contained in~$\mathcal{H}^\alpha$.
As a consequence, since~$\mathcal{H}^\alpha$ is commutative,
the question of the rationality of~$\sum_{k=0}^\infty T(q^k)t^k$
is now well-posed.
We prove in Section~\ref{rationalitysec} that~$\sum_{k=0}^\infty T(q^k)t^k$
is indeed a rational function in~$t$, and we explicitly describe
polynomials~$P(t)$ and~$Q(t)$ with coefficients from~$\mathcal{H}^\alpha$
such that~$\sum_{k=0}^\infty T(q^k)t^k = P(t)Q(t)^{-1}$.

Our second main result determines the structure of~$\mathcal{H}^\alpha$
and the relation of~$\mathcal{H}^\alpha$ to~$\mathcal{H}$.
In Section~\ref{intermediatealgsec} we prove that~$\mathcal{H}^\alpha$
is commutative and is generated by three elements that are algebraically
independent over~$\Z$.
Finally, in Section~\ref{quaternionalgsec}, we show that the
center~$Z(\mathcal{H})$, the commutative algebra~$\mathcal{H}^\alpha$,
and~$\mathcal{H}$ satisfy a familiar pattern. Namely, we define
a graded commutative polynomial ring~$R$ (isomorphic to~$\Z[2^{-1}] \otimes_\Z Z(\mathcal{H})$),
a graded~$R$-algebra~$S$ with involution~$\sigma \colon S \to S$
(isomorphic to~$\Z[2^{-1}] \otimes_\Z \mathcal{H}^\alpha$), and
an element~$z$ in~$R$ such that~$\Z[2^{-1}] \otimes_\Z \mathcal{H}$
is isomorphic to the graded quaternion algebra~$A$ over~$R$ defined by~$\sigma$,~$z$, and~$S$;
we also prove that~$A$ is split. As a consequence, we obtain
a concrete embedding of~$\Z[2^{-1}] \otimes_\Z \mathcal{H}$ into~$\Mat(2,R)$.

An important underpinning of our work is
the calculation of  explicit multiplication tables
involving the generators of~$\mathcal{H}^\alpha$ and elements of the form~$C+\alpha(C)$
for~$C$ a double coset in~$K \backslash \Delta / K$: see Table~\ref{XHptable} and Table~\ref{YHptable}.
These tables in turn depend on Table~2 and Table~3 of~\cite{JLPR}
which are repeated below as Table~\ref{phatable1} and Table~\ref{phatable2}.

We expect that our detailed results about $\mathcal{H}$ may provide a useful example for the general study of the Hecke algebras corresponding to parahoric subgroups of other reductive algebraic groups defined over a local field. Additionally, the explicit results of this paper will be useful for computational investigations of paramodular Siegel modular forms. The authors have no conflicts of interest, generative AI was not used in the research or writing of this paper, and there are no data associated with this work.

\section{Notation and previous results}
\label{sec2}
\label{notationandpreviousresultssec}
In this section we recall some needed definitions and results from our previous paper~\cite{JLPR}.
Let~$F$ be a fixed nonarchimedean field of characteristic zero,
let~$\OF$ be the ring of integers of~$F$,
let~$\p$ be the prime ideal of~$\OF$,
let~$\varpi$ be a generator of~$\p$,
and let~$\OF/\p$ have order~$q$.
If~$x \in F^\times$, then~$x$ has a unique expression~$x=u \varpi^k$
where~$u \in \OF^\times$ and~$k \in \Z$; we define~$v(x) =k$.
Let
\begin{equation}
J=
\begin{bsmallmatrix}
&&1&\\
&&&1\\
-1&&&\\
&-1&&
\end{bsmallmatrix},
\end{equation}
and let~$G=\GSp(4,F)$ be the group of~$g \in \GL(4,F)$ for
which there exists~$\lambda \in F^\times$ such that~$\transp{g} J g = \lambda J$;
such a~$\lambda$ is uniquely determined and will be
denoted by~$\lambda (g)$.
We regard~$G$ as a group of td-type (see~\cite{Cart}).
We denote the center of~$G$ by~$Z$;~$Z$ consists of the
elements~$\mathrm{diag}(z,z,z,z)$ for~$z \in F^\times$.
We let~$G_1=\SSp(4,F)$ be the subgroup of~$g \in G$ with~$\lambda(g)=1$.
Let
\begin{equation}\label{delta}
\Delta =  G \cap
\begin{bsmallmatrix}
\OF & \OF & \p^{-1} & \OF\\
\p & \OF & \OF & \OF\\
\p & \p & \OF & \p \\
\p & \OF & \OF & \OF
\end{bsmallmatrix}.
\end{equation}
The set~$\Delta$ contains~$1$ and is a semi-group contained in~$G$.
If~$g \in \Delta$, then a calculation shows that~$\det(g) \in \OF$
so that~$\lambda(g) \in \OF$ and thus~$v(\lambda(g)) \geq 0$.
We define
\begin{equation}
\label{Kdefeq}
K = \left\{ k \in \Delta \colon \lambda(k) \in \OF^\times\right\}
=\left\{ k \in \Delta \colon v(\lambda(k)) =0 \right\}.
\end{equation}
Then $K$ is a maximal compact open subgroup of $G$, and we refer
to $K$ as the \emph{paramodular subgroup} of $G$ of level $\p$.

We let~$\mathcal{H}$ be the ring~$\mathcal{R}(K,\Delta)$
as defined in Sect.~3.1, p.~54 of \cite{S1}.
As a set,~$\mathcal{H}$ consists of the formal
sums~$\sum_{C \in K \backslash \Delta / K} j(C) C$
with~$j(C) \in \Z$ and~$j(C)=0$ for all but finitely many~$C \in K \backslash \Delta / K$.
The addition for~$\mathcal{H}$ is defined in the obvious fashion.
The multiplication for~$\mathcal{H}$ is described in Section~2 of~\cite{JLPR};
note that the double coset~$K\cdot 1 \cdot K =K$ is the multiplicative identity of~$\mathcal{H}$.
The ring~$\mathcal{H}$ is naturally a~$\Z$-algebra, and we refer to~$\mathcal{H}$ as the
\emph{paramodular Hecke algebra}.
The~$\Z$-algebra~$\mathcal{H}$ admits a natural grading
\begin{equation}
\mathcal{H} = \oplus_{k=0}^\infty \mathcal{H}_k, \qquad
\mathcal{H}_k \cdot \mathcal{H}_j \subset \mathcal{H}_{k+j}, \quad k,j \in
\Z_{\geq 0}.
\end{equation}
Here, for~$k \in \Z_{\geq 0}$,~$\mathcal{H}_k$ is the~$\Z$-linear span in~$\mathcal{H}$
of the~$KgK$ with~$g \in \Delta$ and~$v(\lambda(g))=k$.

The elements of~$K \backslash \Delta / K$ admit a convenient description.
Let
\begin{align}
D_\Delta & =  \left\{ (a,b,c) \in \Z^3 \colon 0 \leq a,\  0 \leq b,\  0 \leq c-a,\  0 \leq c-b \right\}, \\
S_\Delta &= \left\{ (a,b,c) \in D_\Delta \colon  a \leq c-a,\  b \leq c-b\right\}.
\end{align}
The sets~$D_\Delta$ and~$S_\Delta$ are monoids under addition.
For~$a,b,c \in \Z$, let~$d(a,b,c) =\mathrm{diag}(\varpi^a,\varpi^b,\varpi^{c-a},\varpi^{c-b})$;
evidently,~$d(a,b,c) \in G$.
We have~$d(a,b,c) \in \Delta$ if and only if~$(a,b,c) \in D_\Delta$.
We also define~$T(g) = KgK$ for~$g \in \Delta$, and
for~$(a,b,c) \in D_\Delta$ we  abbreviate~$T(a,b,c) = T(d(a,b,c))$;
note that~$T(a,b,c) \in \mathcal{H}_c$. Define
$$
w=\begin{bsmallmatrix}&1&&\\\varpi&&&\\&&&\varpi\\&&1&\end{bsmallmatrix}.
$$
The element~$w$ is in~$\Delta$,~$\lambda(w)=\varpi$, and~$w$ normalizes~$K$.
We define
\begin{equation}
\label{Vdefeq}
V = T(w) = K w K.
\end{equation}
If~$C$ is an element of~$K \backslash \Delta / K$,
then there exist unique~$(a,b,c) \in S_\Delta$ and~$\delta \in \{0,1\}$
such that
\begin{equation}
\label{uniquerepeq}
C = K w^\delta d(a,b,c) K = V^\delta T(a,b,c).
\end{equation}
See Lemma~3.1 and~(27) of~\cite{JLPR}.

The work~\cite{JLPR} also describes~$\mathcal{H}$ in terms of generators and
relations. Define
\begin{equation}
\label{generatorsH}
\qquad X=T(0,0,1),\qquad Y_1=T(0,1,2),\qquad Y_2=T(1,0,2).
\end{equation}
Theorem~3.10 of~\cite{JLPR} proves that~$\mathcal{H}$
is the~$\Z$-algebra generated by~$V$,~$X$,~$Y_1$, and~$Y_2$
subject to seven relations (see~(72) through~(78) of~\cite{JLPR}).
We note that~$V$ and~$X$ are in~$\mathcal{H}_1$ and~$Y_1$
and~$Y_2$ are in~$\mathcal{H}_2$.
Two key results lie behind the description in~\cite{JLPR} of~$\mathcal{H}$ in
terms of generators and relations.
These two theorems calculate~$XT(a,b,c)$ and~$Y_1 T(a,b,c)$
for~$(a,b,c)$ in~$S_\Delta$.
See Theorem~\ref{3p5theorem} and Theorem~\ref{3p6theorem} below.

\begin{theorem}[Theorem~3.5 of~\cite{JLPR}]
\label{3p5theorem}
Let~$(a,b,c) \in S_\Delta$.
There exist unique  integers~$n_1$,~$n_2$,~$n_3$,~$n_4$, and~$n_5$ such that
\begin{align}
&XT(a,b,c)=
n_1 T(a,b,c+1)+n_2 T(a,b+1,c+1)\nonumber \\
&\qquad +n_3T(a+1,b,c+1)+n_4T(a+1,b+1,c+1)
+n_5 VT(a,b,c) \label{degree1multiplicationthmeq1000}
\end{align}
and
for~$1 \leq i \leq 5$,~$n_i  =0$ if~$(e,f,g) \notin S_\Delta$,
where~$(e,f,g)$ is the triple of integers in the~$i$-th summand in
\eqref{degree1multiplicationthmeq1000}.
The integers~$n_i$ are as in Table~\ref{phatable1}.

\begin{table}[h!]
\caption{The integers~$n_i$ from \eqref{degree1multiplicationthmeq1000}}
\label{phatable1}
\renewcommand{\arraystretch}{1}
\begin{tabular}{cllllll}
\toprule
\multicolumn{2}{c}{conditions} & $n_1$ & $n_2$ & $n_3$ & $n_4$ & $n_5$ \\
\midrule
$a<b$ & $b=c-b$ & $1$ & $0$ & $q^2$ & $0$ & $q^2-1$ \\
& $b+1=c-b$ & $1$ & $q+1$ & $q^2$ & $q^3+q^2$ & $q^2-1$ \\
& $b+2 \leq c-b$ & $1$ & $q$ & $q^2$ & $q^3$ & $q^2-1$ \\
\midrule
$a=b$ & $b = c-b$ & $1$ & $0$ & $0$ & $0$ & $0$ \\
 & $b+1 = c-b$ & $1$ & $q+1$ & $q+1$ & $q^3+2q^2+q$ & $q-1$ \\
 & $b+2 \leq c-b$ & $1$ & $q$ & $q$ & $q^3$ & $q-1$ \\
 \midrule
$a>b$ & $a=c-a$  & $1$ & $q^2$ & $0$ & $0$ & $q^2-1$ \\
& $a+1 = c-a$ & $1$ & $q^2$ & $q+1$ & $q^3+q^2$ & $q^2-1$ \\
& $a+2 \leq c-a$ & $1$ & $q^2$ & $q$ & $q^3$ & $q^2-1$ \\
\bottomrule
\end{tabular}
\end{table}
\end{theorem}
\begin{theorem}[Theorem~3.6 of~\cite{JLPR}]
\label{3p6theorem}
Let~$(a,b,c) \in S_\Delta$.
There exist unique  integers~$m_i$,~$1 \leq i \leq 6$, such that
\begin{align}
&Y_1 T(a,b,c) = m_1 T(a,b+1,c+2)+m_2 T(a+1,b+1,c+2)\nonumber \\
&\qquad+ m_3 T(a+2,b+1,c+2) +m_4 V T(a,b,c+1)\nonumber \\
&\qquad+ m_5 V T(a+1,b,c+1) +m_6 V T(a+1,b+1,c+1)\label{degree2multiplicationthmeq1}
\end{align}
and for~$1 \leq i \leq 6$,~$m_i=0$ if~$(e,f,g) \notin S_\Delta$,
where~$(e,f,g)$ is the triple of integers occurring in the~$i$-th summand of
\eqref{degree2multiplicationthmeq1}.
The integers~$m_i$  are  as in Table~\ref{phatable2}.

\begin{table}[h!]
\caption{The integers~$m_i$ from~\eqref{degree2multiplicationthmeq1}}
\label{phatable2}
\renewcommand{\arraystretch}{1}
\begin{tabular}{lcrrrrrr}
\toprule
\multicolumn{2}{c}{conditions} & $m_1$ & $m_2$ & $m_3$ & $m_4$ & $m_5$ & $m_6$\\
\midrule
$a<b$ & $b=c-b$ $\mathrm{and}$ $a+2 = c-a$   & $1$ & $q^3-q^2$ & $q^4+q^3$ & $0$ & $q^3-q^2$ & $0$\\
& $b=c-b$ $\mathrm{and}$ $a+2 < c-a$ & $1$ & $q^3-q^2$ & $q^4$ & $0$ & $q^3-q^2$  & $0$ \\
& $b+1=c-b$ & $1$ & $q^3-q^2$ & $q^4$ & $0$ & $q^3-q^2$ & $q^4-q^2$ \\
& $b+2 \leq c-b$ & $1$ & $q^3-q^2$ & $q^4$ & $0$ & $q^3-q^2$ & $q^4-q^3$ \\
\midrule
$a=b$ & $a=c-a$ & $1$ & $0$ & $0$ & $0$ & $0$ & $0$ \\
& $a+1=c-a$ & $1$ & $q^2$ & $0$ & $0$ & $q^2-1$ & $0$\\
& $a+2=c-a$ & $1$ & $q^2-q$ & $q^3+q^2$ & $0$ & $q^2-q$ & $0$ \\
& $a+3 \leq c-a$ & $1$ & $q^2-q$ & $q^3$ & $0$ & $q^2-q$ & $0$\\
\midrule
$a>b$ & $a=c-a$ & $q$ & $0$ & $0$ & $q-1$ & $0$ & $0$ \\
& $a+1=c-a$ & $q$ & $q^2$ & $0$ & $q-1$ & $q^2-1$ & $0$ \\
& $a+2 =c-a$ & $q$ & $q^2-q$ & $q^3+q^2$ & $q-1$ & $q^2-q$ & $0$ \\
& $a+3 \leq c-a$ & $q$ & $q^2-q$ & $q^3$ & $q-1$ & $q^2-q$ & $0$ \\
\bottomrule
\end{tabular}
\end{table}
\end{theorem}

In this paper we study certain elements of~$\mathcal{H}$.
For~$k \in \Z_{\geq 0}$ we define
\begin{equation}
\label{Tqkdefeq}
T(q^k) =  \sum_{\substack{C \in K \backslash \Delta / K, \\ C \in \mathcal{H}_k}} C.
\end{equation}
By the unique decomposition~\eqref{uniquerepeq} mentioned
above, we may write~$T(q^k)$ for~$k \in \Z_{\geq 0}$ as
\begin{equation}
\label{Tqkformalt1eq}
T(q^k)
= \sum_{(a,b,k) \in S_\Delta} T(a, b, k)
+ V \sum_{(a,b,k-1) \in S_\Delta} T(a,b,k-1).
\end{equation}
In this formula, the second summand is absent if~$k=0$.
For example,~$T(q^0)=1$ and~$T(q^1)= X+V$.

Finally,~\cite{JLPR} also describes an important automorphism
and anti-automorphism of~$\mathcal{H}$.
If ~$\gamma \colon G \to G$ is an automorphism (anti-automorphism)
such that~$\gamma(\Delta)=\Delta$ and~$\gamma(K)=K$,
then the map~$\mathcal{H} \to \mathcal{H}$ determined
by~$KgK \mapsto K \gamma (g) K$ for~$g \in \Delta$ is a
$\Z$-algebra automorphism (anti-automorphism).
We let~$\alpha \colon \mathcal{H} \to \mathcal{H}$
and~$\beta \colon \mathcal{H} \to \mathcal{H}$ denote the automorphism
and anti-automorphism determined by conjugation by~$w$ and the
anti-automorphism of~$G$ that
sends~$\begin{bsmallmatrix} a&b \\ c&d \end{bsmallmatrix}$
to~$\begin{bsmallmatrix} \transp{d} & -\transp{b} \\
-\transp{c} & \transp{a} \end{bsmallmatrix}$, respectively.
Then~$\alpha^2=\beta^2=\mathrm{Id}$,
\begin{equation}
\label{alphabetapropeq}
\alpha \left( T(a,b,c)\right)=T(b,a,c), \qquad \beta \left( T(a,b,c)\right)
=T(c-a,c-b,c)
\end{equation}
for $(a,b,c) \in D_\Delta$, and
\begin{gather}
\alpha( V )=V, \qquad \alpha (X)
=X, \qquad \alpha (Y_1)=Y_2, \qquad \alpha(Y_2)=Y_1,\label{alphageneq}\\
\beta(V) = V, \qquad \beta(X) = X, \qquad \beta(Y_1) = Y_1, \qquad
\beta(Y_2)=Y_2. \label{betageneq}
\end{gather}
We will study the fixed points of~$\alpha$ in the next section.

\section{An intermediate algebra}
\label{sec3}
\label{intermediatealgsec}

One of the goals of this paper is to prove the rationality
result Theorem~\ref{generalrationalitytheorem} concerning the Hecke operators~$T(q^k)$
for~$k \in \Z_{\geq 0}$ (as defined in~\eqref{Tqkdefeq}). The proof of this theorem
is based on the following insight. Let~$k \in \Z_{\geq 0}$.
We may rewrite~\eqref{Tqkdefeq} as
\begin{align}
T(q^k)
&= \sum_{\substack{(a,b,k) \in S_\Delta, \\ a <b} } \left( T(a, b, k) + T(b,a,k) \right)
+ \sum_{\substack{(a,b,k) \in S_\Delta, \\ a =b} }  T(a, b, k) \nonumber \\
&\quad + V \left( \sum_{\substack{(a,b,k-1) \in S_\Delta, \\ a <b} } \left( T(a, b, k-1) + T(b,a,k-1) \right)
+ \sum_{\substack{(a,b,k-1) \in S_\Delta, \\ a =b} }  T(a, b, k-1) \right). \nonumber
\end{align}
This becomes more compact with some new notation.
If~$(a,b,c) \in S_\Delta$,
then we  define
\begin{equation}
 T^\alpha (a,b,c) =
\begin{cases}
T(a,b,c)=T(a,a,c) &\text{if $a=b$},\\
T(a,b,c)+T(b,a,c)&\text{if $a \neq b$}.
\end{cases}
\end{equation}
We also define
\begin{equation}
\label{sdeltaplusdefeq}
S_\Delta^+ = \{ (a,b,c) \in \Z^3 \colon 0 \leq a \leq b \leq c-b \leq c-a \}
= \{ (a,b,c) \in \Z^3 \colon 0 \leq a \leq b \leq \lfloor c/2 \rfloor \}.
\end{equation}
With this new notation we now have
\begin{equation}
\label{Tqkalphaeq}
T(q^k)
= \sum_{(a,b,k) \in S_\Delta^+ } T^\alpha (a, b, k)
 + V  \sum_{(a,b,k-1) \in S_\Delta^+ } T^\alpha (a, b, k-1).
\end{equation}
Moreover, if~$(a,b,c) \in S_\Delta^+$, then by~\eqref{alphabetapropeq}
the element~$T^\alpha(a,b,c)$ is contained in
\begin{equation}\label{Halphadef}
\mathcal{H}^{\alpha}=\{A\in\mathcal{H}:\alpha (A) =A\}.
\end{equation}
Here,~$\alpha$ is the involutive automorphism from Section~\ref{notationandpreviousresultssec}.
Evidently, the set~$\mathcal{H}^\alpha$ is a~$\Z$-subalgebra of~$\mathcal{H}$.
Since~$\alpha (V)=V$, the element~$V$ is also in~$\mathcal{H}^\alpha$.
Thus, by~\eqref{Tqkalphaeq}, ~$T(q^k)$ is contained in~$\mathcal{H}^\alpha$.
In this section we will prove
that~$\mathcal{H}^\alpha$ is generated by the elements~$V$,~$X$, and
\begin{equation}
\label{Ydefeq}
Y = Y_1 + Y_2,
\end{equation}
that~$\mathcal{H}^\alpha$ is commutative,
and that~$V$,~$X$, and~$Y$ are algebraically independent over~$\Z$.
To end the section we will give formulas for~$X T^\alpha(a,b,c)$
and~$Y T^\alpha(a,b,c)$ for~$(a,b,c) \in S_\Delta^+$.
These formulas will be used to prove Theorem~\ref{generalrationalitytheorem}.

\begin{lemma}
\label{Halphacharlemma}
Let~$A \in \mathcal{H}$.
Then~$A \in \mathcal{H}^\alpha$ if and only if~$A$ has the form
\begin{equation}
\label{Halphacharlemmaeq1}
A = \sum_{k=0}^\infty\sum_{\substack{e_1,e_2,j \in \Z_{\geq 0},\\ e_1 +e_2 +2j=k}}
c(e_1,e_2,j)V^{e_1} X^{e_2} (Y_1^j + Y_2^j)
\end{equation}
where the~$c(e_1,e_2,j)$ are in~$\Z$ and are zero for all but finitely
many~$(e_1,e_2,j)$.
\end{lemma}

\begin{proof}
Let~$A\in\mathcal{H}^\alpha$.
By Theorem~3.10 of~\cite{JLPR}
\begin{equation}
\label{Halphacharlemmaeq2}
A= \sum_{k=0}^\infty
\sum_{\substack{e_1,e_2,e_3,e_4 \in \Z_{\geq 0},\\ e_1+e_2+2e_3+2e_4 =k,\\
e_3e_4=0}}
n(e_1,e_2,e_3,e_4) V^{e_1} X^{e_2} Y_1^{e_3} Y_2^{e_4}
\end{equation}
where the integers~$n(e_1,e_2,e_3,e_4)$ are unique and
are zero for all but finitely many~$(e_1,e_2,e_3,e_4)$.
Since~$A$ is fixed by~$\alpha$,~\eqref{alphageneq} implies that
\begin{equation}
\label{Halphacharlemmaeq3}
A =
\sum_{k=0}^\infty
\sum_{\substack{e_1,e_2,e_3,e_4 \in \Z_{\geq 0},\\ e_1+e_2+2e_3+2e_4 =k,\\
e_3e_4=0}}
n(e_1,e_2,e_3,e_4) V^{e_1} X^{e_2} Y_2^{e_3} Y_1^{e_4}.
\end{equation}
By the uniqueness of expression in~\eqref{Halphacharlemmaeq2},
we have~$n(e_1,e_2,e_3,e_4)=n(e_1,e_2,e_4,e_3)$ for all
indices~$(e_1,e_2,e_3,e_4)$.
Therefore,
\begin{equation}
\label{Halphacharlemmaeq4}
A =
\sum_{k=0}^\infty
\sum_{\substack{e_1,e_2,j \in \Z_{\geq 0},\\ e_1+e_2+2j =k}}
n(e_1,e_2,j,0) V^{e_1} X^{e_2} (Y_1^j + Y_2^j).
\end{equation}
This proves the existence of an expression for~$A$ as in~\eqref{Halphacharlemmaeq1}.
If~$A$ has an expression as in~\eqref{Halphacharlemmaeq1}, then~\eqref{alphageneq}
implies that~$A$ is fixed by~$\alpha$.
\end{proof}

\begin{proposition}
\label{comalglemma}
The subalgebra~$\mathcal{H}^\alpha$ is commutative
and is generated by~$V, X,$ and~$Y$.
\end{proposition}

\begin{proof}
The elements~$V$,~$X$, and~$Y=Y_1+Y_2$ commute by~(7.2),~(7.3), and~(7.4) of Lemma~3.8
of~\cite{JLPR}.
Therefore, to prove the proposition, it will suffice
to prove that~$\mathcal{H}^\alpha$ is generated by~$V$,~$X$, and~$Y$.
Let~$\mathcal{H}'$ be the subalgebra of~$\mathcal{H}$ generated by~$X, V$ and~$Y$.
We need to prove that~$\mathcal{H}^\alpha$ is contained in~$\mathcal{H}'$.
For~$k \in \Z_{\geq 0}$ let~$\mathcal{H}_k^\alpha = \mathcal{H}^\alpha \cap \mathcal{H}_k$.
Since~$\mathcal{H}^\alpha = \oplus_{k=0}^\infty \mathcal{H}_k^\alpha$
because~$\alpha$ preserves~$\mathcal{H}_k$ for~$k \in \Z_{\geq 0}$
(see~\eqref{alphageneq}), to prove that~$\mathcal{H}^\alpha$
is contained in~$\mathcal{H}'$, it will suffice to prove that~$\mathcal{H}_k^\alpha$
is contained in~$\mathcal{H}'$ for all~$k \in \Z_{\geq 0}$.
We will show by induction on~$k$ that~$\mathcal{H}^\alpha_k$ is in~$\mathcal{H}'$
for~$k \in \Z_{\geq 0}$.
This is clear for $k=0$.
Suppose that~$k \in \Z_{\geq 0}$ and that~$\mathcal{H}^\alpha_i\subset \mathcal{H}'$ for $0 \leq i \leq k$; we will prove that~$\mathcal{H}^\alpha_{k+1}\subset \mathcal{H}'$.
Let~$A \in \mathcal{H}^\alpha_{k+1}$ so that~$A\in\mathcal{H}_{k+1}$ and~$\alpha(A) =A$.
By Lemma~\ref{Halphacharlemma}
we may write
$$
A = \sum_{\substack{e_1,e_2,j \in \Z_{\geq 0},\\ e_1 +e_2 +2j=k+1}}
c(e_1,e_2,j)V^{e_1} X^{e_2} (Y_1^j + Y_2^j)
$$
where the~$c(e_1,e_2,j)$ are in $\Z$.
Since~$V,X \in \mathcal{H}'$, to prove that~$A \in \mathcal{H}'$
it will suffice to prove that~$Y_1^j + Y_2^j \in \mathcal{H}'$ for~$0 \leq 2j \leq k+1$.
If~$j \in \Z_{\geq 0}$ is such that~$0 \leq 2j < k+1$, then~$Y_1^j + Y_2^j \in \mathcal{H}'$ by
the induction hypothesis.
Assume that~$j \in \Z_{\geq 0}$ is such that~$2j = k+1$.
We may assume that~$j\geq 2$, since the case that~$j=1$ is trivial.
Now
\begin{equation}
\label{comalglemmaeq1}
Y^j=(Y_1+Y_2)^j = Y_1^j + Y_2^j +B
\end{equation}
where~$B$ is a sum of terms of the form $\cdots Y_1Y_2 \cdots$ or $\cdots Y_2Y_1 \cdots$.
Using~(73),~(74),~(77), and~(78) of \cite{JLPR},
we see that~$B = VB_0$ with~$B_0 \in \mathcal{H}_k$.
Since~$(Y_1+Y_2)^j-Y_1^j-Y_2^j$ is invariant under~$\alpha$ by~\eqref{alphageneq},
it follows that~$\alpha(B)=B$.
Therefore,~$V B_0 = V \alpha(B_0)$.
By Lemma~5.1 of~\cite{JLPR} we have~$\alpha(B_0) = B_0$.
Since~$B_0$ is in~$\mathcal{H}^\alpha_{2j-1} = \mathcal{H}^\alpha_{k}$,
the induction hypothesis now implies that~$B_0 \in \mathcal{H}'$;
therefore,~$B \in \mathcal{H}'$. Since~$Y^j \in \mathcal{H}'$,
we conclude from \eqref{comalglemmaeq1} that~$Y_1^j + Y_2^j$ is contained in~$\mathcal{H}'$.
\end{proof}

\begin{lemma}
\label{Hcircpolylemma}
The elements~$V$,~$X$, and~$Y$ are algebraically independent over~$\Z$.
\end{lemma}
\begin{proof}
Assume that there is some non-zero~$p(x_1,x_2,x_3) \in \Z[x_1,x_2,x_3]$
such that~$p(V,X,Y) =0$;
we will obtain a contradiction.
Choose such a~$p(x_1,x_2,x_3)$ with minimal degree, and write
$$
p(x_1,x_2,x_3) = p_1(x_2,x_3) + x_1 p_2(x_1,x_2,x_3)
$$
where~$p_1(x_2,x_3) \in \Z[x_2,x_3]$ and~$p_2(x_1,x_2,x_3) \in \Z[x_1,x_2,x_3]$.
We have that~$p_1(x_2,x_3) \neq 0$ since otherwise,~$p(x_1,x_2,x_3) = x_1 p_2(x_1,x_2,x_3)$,
so that~$p_2(V,X,Y) =0$ by Lemma~5.1 of \cite{JLPR};
this would contradict
the minimality of the degree of~$p(x_1,x_2,x_3)$.
Now let~$z_1,z_2 \in \C$.
By Lemma 4.1 of \cite{JLPR} there exists a homomorphism~$\chi: \mathcal{H} \to \C$
such that~$\chi(V)=0$, $\chi(Y_2)=0$, $\chi(X) =z_1$, and~$\chi(Y_1)=z_2$.
Applying~$\chi$ to
$$
p(V,X,Y_1+Y_2)= p_1(X,Y_1+Y_2) + V p_2(V,X,Y_1+Y_2)=0
$$
we obtain~$p_1(z_1,z_2)=0$. Since this holds for all~$z_1,z_2 \in \C$, we find
that~$p_1(x_2,x_3)=0$, a contradiction.
\end{proof}

\begin{lemma}
\label{Hprimemodlemma}
Regard~$\mathcal{H}$ as an~$\mathcal{H}^\alpha$-module, and let $M$
be the $\mathcal{H}^\alpha$-submodule of~$\mathcal{H}$ spanned
by~$1$ and~$Y_1$. Then~$M = \mathcal{H}$.
\end{lemma}
\begin{proof}
We will prove that~$\mathcal{H}_k \subset M$ for all~$k \in \Z_{\geq 0}$
by induction on~$k$.
It is clear that~$\mathcal{H}_0,\mathcal{H}_1,\mathcal{H}_2 \subset M$.
Assume that~$k \geq 2$ and~$\mathcal{H}_j \subset M$ for~$1 \leq j \leq k$;
we will prove that $\mathcal{H}_{k+1} \subset M$. Let $A \in \mathcal{H}_{k+1}$.
By Theorem~3.10 of~\cite{JLPR} we have
\begin{equation}
\label{Hprimemodlemmaeq1}
A=
\sum_{\substack{e_1,e_2,e_3,e_4 \in \Z_{\geq 0},\\ e_1+e_2+2e_3+2e_4 =k+1,\\
e_3e_4=0}}
n(e_1,e_2,e_3,e_4) V^{e_1} X^{e_2} Y_1^{e_3} Y_2^{e_4}
\end{equation}
where~$n(e_1,e_2,e_3,e_4) \in \Z$ for~$e_1,e_2,e_3,e_4$ as in~\eqref{Hprimemodlemmaeq1}.
From~\eqref{Hprimemodlemmaeq1}, we see that to prove that~$A \in M$
it will suffice to prove that~$Y_1^j,Y_2^j \in M$ for~$0 \leq 2j \leq k+1$.
If~$k+1$ is odd, then~$2j \leq k$,
so that~$Y_1^j,Y_2^j \in M$ for~$0 \leq 2j \leq k+1$ by the induction hypothesis.
Assume that~$k+1$ is even
and set~$j=(k+1)/2$; we need to prove that~$Y_1^j, Y_2^j \in M$.
Since $Y_1^j+Y_2^j \in \mathcal{H}^\alpha \subset M$ by~\eqref{alphageneq}, it will
suffice to prove that $Y_2^j \in M$. Now
\begin{equation}
\label{Hprimemodlemmaeq2}
(Y_1^{j-1}+Y_2^{j-1})Y_2 = Y_1^{j-1}Y_2 +Y_2^j,
\end{equation}
and the induction hypothesis implies that this element is in $M$ (because
$Y_1^{j-1}+Y_2^{j-1} \in \mathcal{H}^\alpha$ and $Y_2 \in M$
by the induction hypothesis, so that $(Y_1^{j-1}+Y_2^{j-1})Y_2 \in M$).
By (78) of \cite{JLPR}, $Y_1^{j-1}Y_2 =VB$ for some $B \in \mathcal{H}$;
evidently, $B \in \mathcal{H}_k$ as $Y_1^{j-1}Y_2 \in \mathcal{H}_{k+1}$.
By the induction hypothesis, $B \in M$. It follows that $Y_1^{j-1}Y_2 =VB \in M$.
From \eqref{Hprimemodlemmaeq2}, we now have $Y_2^j \in M$, as desired.
\end{proof}

\begin{proposition}
\label{Xmultprop}
Let~$(a,b,c) \in S_\Delta^+$. There exist
unique integers~$\tilde n_1$,~$\tilde n_2$,~$\tilde n_3$,~$\tilde n_4$,
and~$\tilde n_5$ such that
\begin{align*}
X  T^\alpha(a,b,c)&  =  \tilde n_1  T^\alpha(a,b,c+1)
+ \tilde n_2  T^\alpha (a,b+1,c+1)\\
&\quad +\tilde n_3  T^\alpha (a+1,b,c+1)
+\tilde n_4  T^\alpha(a+1,b+1,c+1) \\
&\quad +\tilde n_5  V  T^\alpha(a,b,c)
\end{align*}
and~$\tilde n_i = 0$ if~$(e,f,g) \notin S_\Delta^+$, where~$(e,f,g)$ is the triple
of integers in the $i$-th summand. The integers~$\tilde n_1$,~$\tilde n_2$,~$\tilde n_3$,~$\tilde n_4$,
and~$\tilde n_5$ are as in Table~\ref{XHptable}.
\end{proposition}
\begin{proof}
This follows from a calculation applying Theorem~\ref{3p5theorem} to the subalgebra~$\mathcal{H}^\alpha$.
\end{proof}
\begin{table}
\caption{The integers~$\tilde n_i$ from Proposition~\ref{Xmultprop}}
\label{XHptable}
\begin{tabular}{l@{\hskip 0.5cm}l@{\hskip 1cm}l@{\hskip 1cm}l@{\hskip 1cm}l@{\hskip 1cm}l@{\hskip 1cm}l}
\toprule
\multicolumn{2}{c}{conditions}& $\tilde n_1$ & $\tilde n_2$ & $\tilde n_3$ & $\tilde n_4$ & $\tilde n_5$ \\
\midrule
$a=b$ & $b=c-b$        & $1$ & $0$ & $0$ & $0$ & $0$\\
      & $b+1 =c-b$     & $1$ & $q+1$ & $0$ & $q^3+2q^2+q$ & $q-1$ \\
      & $b+2 \leq c-b$ & $1$ & $q$ & $0$ & $q^3$ &  $q-1$ \\
\midrule
$a+1 = b$ & $b=c-b$ & $1$ & $0$ & $2q^2$ & $0$ & $q^2-1$ \\
 & $b+1=c-b$  & $1$ & $q+1$ &  $2q^2$ & $q^3+q^2$ & $q^2-1$ \\
 & $b+2\leq c-b$ & $1$ & $q$ & $2q^2$ & $q^3$ & $q^2-1$ \\
\midrule
$a+1 < b$ & $b=c-b$ & $1$ & $0$ & $q^2$ & $0$ & $q^2-1$ \\
 & $b+1=c-b$  & $1$ & $q+1$ &  $q^2$ & $q^3+q^2$ & $q^2-1$ \\
 & $b+2\leq c-b$ & $1$ & $q$ & $q^2$ & $q^3$ & $q^2-1$ \\
\bottomrule
\end{tabular}
\end{table}

\begin{table}
\caption{The integers~$\tilde m_i$ from Proposition~\ref{Ymultprop}}
\label{YHptable}
\begin{adjustbox}{width=1\textwidth}
\begin{tabular}{lllllllllll}
\toprule
\multicolumn{2}{c}{conditions}& $\tilde m_1$ & $\tilde m_2$ & $\tilde m_3$ & $\tilde m_4$ & $\tilde m_5$ & $\tilde m_6$ & $\tilde m_7$ & $\tilde m_8$ & $\tilde m_9$ \\
\midrule
$a=b$ & $b=c-b$ & $1$ & $0$ & $0$ & $0$ & $0$ & $0$ & $0$ & $0$ & $0$\\
& $b+1=c-b$ & $1$ & $0$ & $2q^2$ & $0$ & $0$ & $0$ & $q^2-1$ & $0$ & $0$ \\
& $b+2=c-b$ & $1$ & $0$ & $2(q^2-q)$ & $q^3+q^2$ & $0$ & $0$ & $q^2-q$ & $0$ & $0$ \\
& $b+3 \leq c-b$ & $1$ & $0$ & $2(q^2-q)$ & $q^3$ & $0$ & $0$ & $q^2-q$ & $0$ & $0$ \\
\midrule
$a+1=b$ & $b=c-b$ & $1$ & $2q$ & $q^3-q^2$ & $0$ & $2(q^4+q^3)$ & $q-1$ & $0$ & $2(q^3-q^2)$ & $0$\\
& $b+1=c-b$ & $1$ & $2q$ & $q^3$ & $0$ & $2q^4$ & $q-1$ & $q^2-1$ & $2(q^3-q^2)$ & $q^4-q^2$ \\
& $b+2=c-b$ & $1$ & $2q$ & $q^3-q$ & $q^3+q^2$ & $2q^4$ & $q-1$ & $q^2-q$ & $2(q^3-q^2)$ & $q^4-q^3$\\
& $b+3 \leq c-b$ & $1$ & $2q$ & $q^3-q$ & $q^3$ & $2q^4$ & $q-1$ & $q^2-q$ & $2(q^3-q^2)$ & $q^4-q^3$\\
\midrule
$a+1 < b$ & $b=c-b$ & $1$ & $q$ & $q^3-q^2$ & $0$ & $q^4$ & $q-1$ & $0$ & $q^3-q^2$ & $0$\\
& $b+1=c-b$ & $1$ & $q$ & $q^3$ & $0$ & $q^4$ & $q-1$ & $q^2-1$ & $q^3-q^2$ & $q^4-q^2$ \\
& $b+2=c-b$ & $1$ & $q$ & $q^3-q$ & $q^3+q^2$ & $q^4$ & $q-1$ & $q^2-q$ & $q^3-q^2$ & $q^4-q^3$ \\
& $b+3 \leq c-b$ & $1$ & $q$ & $q^3-q$ & $q^3$ & $q^4$ & $q-1$ & $q^2-q$ & $q^3-q^2$ & $q^4-q^3$ \\
\bottomrule
\end{tabular}
\end{adjustbox}
\end{table}

\begin{proposition}
\label{Ymultprop}
Let~$(a,b,c) \in S_\Delta^+$. There exist
unique integers~$\tilde m_i$ for~$1 \leq i \leq 9$ such that
\begin{align*}
Y  T^\alpha(a,b,c) &  = \tilde m_1  T^\alpha(a,b+1,c+2)   \\
&  \quad + \tilde m_2  T^\alpha(a+1,b,c+2)\\
&  \quad + \tilde m_3  T^\alpha(a+1,b+1,c+2) \\
&  \quad + \tilde m_4  T^\alpha(a+1,b+2,c+2) \\
&  \quad + \tilde m_5  T^\alpha(a+2,b+1,c+2) \\
&  \quad + \tilde m_6 V  T^\alpha(a,b,c+1) \\
&  \quad + \tilde m_7 V  T^\alpha(a,b+1,c+1) \\
&  \quad + \tilde m_8 V   T^\alpha(a+1,b,c+1) \\
&  \quad + \tilde m_9 V  T^\alpha(a+1,b+1,c+1).
\end{align*}
and~$\tilde m_i = 0$ if~$(e,f,g) \notin S_\Delta^+$, where~$(e,f,g)$ is the triple
of integers in the~$i$-th summand. The integers~$\tilde m_i$ for~$1 \leq i \leq 9$
are as in Table~\ref{YHptable}.
\end{proposition}
\begin{proof}
This follows from a calculation applying Theorem~\ref{3p6theorem} to the subalgebra~$\mathcal{H}^\alpha$.
\end{proof}

\section{Rationality}
\label{sec4}
\label{rationalitysec}

In this section we obtain a rationality result
(Theorem~\ref{generalrationalitytheorem}) for the
series~$\sum_{k=0}^\infty T(q^k)t^k$.
To be more precise,
let~$A=\oplus_{k \geq 0} A_k$ be a graded commutative ring.
Let~$(T_k)_{k \geq 0}$ be a sequence of elements of~$A$ with
the property that~$T_k \in A_k$ for~$k \geq 0$.
We say that~$\left(T_k\right)_{k \geq 0}$ is \emph{recursive}
if there exists a positive integer~$m$,
elements~$c_0,\dots,c_m$ of~$A$ with~$c_i \in A_i$ for~$0 \leq i \leq m$
and~$c_0 \in A^\times$ (i.e.,~$c_0$ is a unit of the ring~$A$), and
$$
\sum_{i=0}^m c_i T_{k-i} =0 \qquad \textnormal{for $k \geq m$}.
$$
If $(T_k)_{k \geq 0}$ is recursive and~$t$ is an indeterminate, then in the power series
ring~$A[[t]]$ we have the identity
\begin{equation}
\label{recideq}
\left( \sum_{i=0}^m c_i t^i \right)
\left( \sum_{k=0}^\infty T_k t^k \right)
= \sum_{k=0}^{m-1} \left( \sum_{i=0}^k c_i T_{k-i} \right) t^k.
\end{equation}
Since~$\sum_{i=0}^m c_i t^i$ is a unit of the ring~$A[[t]]$
because~$c_0 \in A^\times$, it follows that the power series
$\sum_{k=0}^\infty T_k t^k$ is a rational function, i.e., is an element
of the field of rational functions~$A(t) \subset A[[t]]$.
In this section we will prove that the sequence~$\left(T(q^k)\right)_{k \geq 0}$
is recursive.

\begin{lemma}\label{recursionlemma}
For~$k \in \Z_{\geq 0}$ define
$$
B_k=\sum_{b=0}^{\lfloor k/2 \rfloor} T^\alpha(0,b,k).
$$
For~$k \geq 2$ we have
\begin{equation}
T(q^k)  = B_k +V T(q^{k-1}). \label{generalrationalitytheoremeq5}
\end{equation}
\end{lemma}
\begin{proof}
For~$k \in \Z_{\geq 0}$ define the following element of~$\mathcal{H}^\alpha$:
\begin{equation}
\label{generalrationalitytheoremeq2}
T_0(q^k) = \sum_{(a,b,k) \in S_\Delta^+} T^\alpha (a,b,k).
\end{equation}
By~\eqref{Tqkalphaeq} we have, for~$k \geq 1$,
\begin{equation}
\label{generalrationalitytheoremeq3}
T(q^k) = T_0(q^k) + V T_0 (q^{k-1}).
\end{equation}
Let~$k \geq 2$. Then
\begin{align}
T_0(q^k) &= \sum_{(a,b,k) \in S_\Delta^+} T^\alpha (a,b,k) \nonumber \\
&= \sum_{(0,b,k) \in S_\Delta^+} T^\alpha (0,b,k)
+\sum_{\substack{(a,b,k) \in S_\Delta^+,\\ a>0}} T^\alpha (a,b,k)\nonumber \\
&= B_k
+V^2 \sum_{\substack{(a,b,k) \in S_\Delta^+,\\ a>0}} T^\alpha (a-1,b-1,k-2)\nonumber \\
&= B_k
+V^2 \sum_{(a,b,k-2) \in S_\Delta^+} T^\alpha (a,b,k-2)\nonumber \\
& = B_k + V^2 T_0(q^{k-2}). \label{generalrationalitytheoremeq4}
\end{align}
Therefore,
\begin{align}
T(q^k)
& = T_0(q^k) + V T_0(q^{k-1})
\qquad \text{(by~\eqref{generalrationalitytheoremeq3})} \nonumber \\
&= B_k + V^2 T_0(q^{k-2}) + V T_0(q^{k-1})
\qquad \text{(by~\eqref{generalrationalitytheoremeq4})} \nonumber  \\
&= B_k + V^2 T_0(q^{k-2}) + V \left( T(q^{k-1}) - V T_0(q^{k-2}) \right)
\qquad \text{(by~\eqref{generalrationalitytheoremeq3})} \nonumber \\
& = B_k +V T(q^{k-1}). \nonumber
\end{align}
This proves~\eqref{generalrationalitytheoremeq5}.
\end{proof}

\begin{lemma}
\label{00krecursion}
For~$k \in \Z_{\geq 0}$ define
\begin{equation}
\label{akdefeq}
a_k = T^\alpha (0,0,k).
\end{equation}
We have
\begin{equation}
\label{00krecursioneq1}
a_{k+4} =A a_{k+3}+B a_{k+2}+C a_{k+1} +D a_k \qquad \text{for~$k \geq 1$},
\end{equation}
where
\begin{equation}
\label{00krecursioneq2}
\left\{
\begin{array}{l}
A=X+(1-q^2)V,\\
B=q(q-1)VX-(q^3+q)V^2-qY,\\
C=q^3 V^2X+q^3(1-q^2)V^3,\\
D=-q^6V^4
\end{array}
\right\}.
\end{equation}
\end{lemma}
\begin{proof}
Assume that~$k \in \Z$ and~$k \geq 2$. By Proposition~\ref{Xmultprop}
we have
\begin{align*}
X T^\alpha (0,0,k) & =  T^\alpha (0,0,k+1) + q  T^\alpha (0,1,k+1) +
 q^3  T^\alpha (1,1,k+1) + (q-1) V  T^\alpha (0,0,k).
\end{align*}
Solving for $T^\alpha (0,1,k+1)$, we obtain
\begin{align*}
 T^\alpha (0,1,k+1)
& = q^{-1} X  T^\alpha (0,0,k) - q^{-1} T^\alpha (0,0,k+1) - q^2  T^\alpha (1,1,k+1)\\
&\quad -q^{-1} (q-1) V  T^\alpha (0,0,k)\\
& = q^{-1} X  T^\alpha (0,0,k) - q^{-1} T^\alpha (0,0,k+1) - q^2 V^2  T^\alpha (0,0,k-1)\\
&\quad -q^{-1} (q-1) V  T^\alpha (0,0,k)\\
& = - q^{-1} T(0,0,k+1) +q^{-1} \left(X - (q-1) V\right) T^\alpha (0,0,k) \\
&\quad - q^2 V^2  T^\alpha (0,0,k-1)\\
& = -q^{-1} a_{k+1} +  q^{-1} \left( X-(q-1)V \right)  a_k  -q^2 V^2 a_{k-1}.
\end{align*}
In summary,
\begin{equation}
\label{00krecursioneq134}
 T^\alpha (0,1,k+1)
= a a_{k+1} + b a_k + c a_{k-1} \qquad \text{for $k \geq 2$},
\end{equation}
where
$$
a = -q^{-1}, \qquad b = q^{-1} \left( X-(q-1)V \right) , \qquad c = -q^2 V^2.
$$
Further assume that~$k \geq 3$. Then by Proposition~\ref{Ymultprop},
\begin{align*}
Ya_k & = Y T^\alpha (0,0,k)\\
& =  T^\alpha (0,1,k+2) + 2(q^2-q)  T^\alpha (1,1,k+2) \\
&  \quad + q^3  T^\alpha (1,2,k+2)  + (q^2-q) V  T^\alpha (0,1,k+1) \\
& =  T^\alpha (0,1,k+2) + 2(q^2-q) V^2  T^\alpha (0,0,k) \\
&  \quad + q^3 V^2  T^\alpha (0,1,k)  + (q^2-q) V  T^\alpha (0,1,k+1) \\
& =  a a_{k+2} + b a_{k+1} + c a_{k} \\
&\quad+ 2(q^2-q) V^2 a_k \\
&\quad+ q^3 V^2 \left( a a_{k} + b a_{k-1} + c a_{k-2} \right) \\
&\quad+ (q^2-q) V \left( a a_{k+1} + b a_k + c a_{k-1} \right)\\
&= a a_{k+2}\\
&\quad+ \left(  (q^2-q) V a+  b \right) a_{k+1} \\
&\quad+ \left(2(q^2-q) V^2+ q^3 V^2 a + (q^2-q) V b + c \right) a_k\\
&\quad+ \left(q^3 V^2 b + (q^2-q) V c  \right) a_{k-1}\\
&\quad+ q^3 V^2 c a_{k-2}.
\end{align*}
Solving for~$aa_{k+2}$, we obtain:
\begin{align*}
a a_{k+2}
&=Y a_k \\
&\quad- \left( (q^2-q) V a+  b \right) a_{k+1} \\
&\quad- \left(2(q^2-q) V^2+ q^3 V^2 a + (q^2-q) V b + c \right) a_k\\
&\quad- \left(q^3 V^2 b + (q^2-q) V c  \right) a_{k-1}\\
&\quad- q^3 V^2 c a_{k-2}\\
&=Y a_k \\
&\quad+ \left( -(q^2-q) V a-  b \right) a_{k+1} \\
&\quad+ \left(-2(q^2-q) V^2 -q^3 V^2 a - (q^2-q) V b - c \right) a_k\\
&\quad+ \left(-q^3 V^2 b - (q^2-q) V c  \right) a_{k-1}\\
&\quad- q^3 V^2 c a_{k-2}\\
&= \left( -(q^2-q) V a-  b \right) a_{k+1} \\
&\quad+ \left(Y-2(q^2-q) V^2 -q^3 V^2 a - (q^2-q) V b - c \right) a_k\\
&\quad+ \left(-q^3 V^2 b - (q^2-q) V c  \right) a_{k-1}\\
&\quad- q^3 V^2 c a_{k-2}.
\end{align*}
Recalling that~$a=-q^{-1}$, multiplying by~$-q$ yields:
\begin{align*}
a_{k+2}&= \left( q(q^2-q)  V a+ q  b \right) a_{k+1} \\
&\quad+ \left(-qY +2q(q^2-q) V^2+q^4 V^2 a +q (q^2-q) V b +q c \right) a_k\\
&\quad+ \left(q^4 V^2 b +q(q^2-q) V c  \right) a_{k-1}\\
&\quad+ q^4 V^2 c a_{k-2} \\
& = A a_{k+1} +B a_k + C a_{k-1} + D a_{k-2}.
\end{align*}
This expression for~$a_{k+2}$ holds for~$k \geq 3$ and implies~\eqref{00krecursioneq1}.
\end{proof}

\begin{lemma}
\label{Bkrecursion}
If~$k\geq 2$, then
\begin{align}
& B_{k+2} +(q+1)^{-1}((q^2-1) V-X) B_{k+1} + q^2 V^2 B_{k} \nonumber \\
&\qquad = (q+1)^{-1} q a_{k+2} +(q+1)^{-1}(q^2-q) V a_{k+1}  - (q+1)^{-1} q^2 V^2 a_{k}.
\label{Bkrecursioneq101}
\end{align}
\end{lemma}
\begin{proof}
For~$k\in\Z_{\ge 0}$ define
$$
C_k=\sum_{a=0}^{\lfloor k/2\rfloor}T(a,0,k).
$$
Then
$$
B_k=C_k+\alpha(C_k)-T(0,0,k)=C_k+\alpha(C_k)-a_k,
$$
where~$a_k$ is as in Lemma \ref{00krecursion}.
Let~$k\geq 3$. Let~$\delta_k=0$ if~$k$ is even and~$\delta_k=1$ if~$k$ is odd.
Let~$a\in \Z_{\geq 0}$ be such that~$0\le a\leq \lfloor k/2\rfloor$.
We have
\begin{align*}
XT(a,0,k)=&n_1(a,0,k)T(a,0,k+1)+n_2(a,0,k)T(a,1,k+1)\\
&+n_3(a,0,k)T(a+1,0,k+1)+n_4(a,0,k)T(a+1,1,k+1)\\
&+n_5(a,0,k)VT(a,0,k).
\end{align*}
 It follows that
\begin{align*}
XC_k=&\sum_{a=0}^{\lfloor k/2\rfloor}n_1(a,0,k)T(a,0,k+1)
+\sum_{a=0}^{\lfloor k/2\rfloor}n_2(a,0,k)T(a,1,k+1)
\\
&+\sum_{a=0}^{\lfloor k/2\rfloor}n_3(a,0,k)T(a+1,0,k+1)
+\sum_{a=0}^{\lfloor k/2\rfloor}n_4(a,0,k)T(a+1,1,k+1)
\\
&+V\sum_{a=0}^{\lfloor k/2\rfloor}n_5(a,0,k)T(a,0,k).
\end{align*}
Using Table~\ref{phatable1} we now calculate each of the five terms in this sum.
We have:
\begin{align*}
&\sum_{a=0}^{\lfloor k/2\rfloor} n_1(a,0,k)T(a,0,k+1)
\\
&\qquad =\left(\sum_{a=0}^{\lfloor (k+1)/2\rfloor}T(a,0,k+1)\right)
-\delta_kT(\lfloor (k+1)/2\rfloor,0,k+1)
\\
&\qquad =C_{k+1}-\delta_kT(\lfloor (k+1)/2\rfloor,0,k+1).
\end{align*}
And:
\begin{align*}
&\sum_{a=0}^{\lfloor k/2\rfloor}n_2(a,0,k)T(a,1,k+1)
\\
&\qquad =n_2(0,0,k)T(0,1,k+1)
+\sum_{a=1}^{\lfloor k/2\rfloor}n_2(a,0,k)T(a,1,k+1)
\\
&\qquad =qT(0,1,k+1)+q^2\sum_{a=1}^{\lfloor k/2\rfloor}T(a,1,k+1)
\\
&\qquad =qT(0,1,k+1)+q^2V^2\sum_{a=1}^{\lfloor k/2\rfloor}T(a-1,0,k-1)
\\
&\qquad =qT(0,1,k+1)+q^2V^2\sum_{a=0}^{\lfloor k/2\rfloor-1}T(a,0,k-1)
\\
&\qquad =qT(0,1,k+1)-\delta_kq^2V^2T(\lfloor (k-1)/2\rfloor,0,k-1)
+q^2V^2\sum_{a=0}^{\lfloor (k-1)/2\rfloor}T(a,0,k-1)
\\
&\qquad =qT(0,1,k+1)-\delta_kq^2V^2T(\lfloor (k-1)/2\rfloor,0,k-1)+q^2V^2C_{k-1}.
\end{align*}
And:
\begin{align*}
&\sum_{a=0}^{\lfloor k/2\rfloor}n_3(a,0,k)T(a+1,0,k+1)
\\
&\qquad=n_3(\lfloor k/2 \rfloor,0,k)T(\lfloor k/2 \rfloor+1,0,k+1)+\sum_{a=0}^{\lfloor k/2\rfloor-1}n_3(a,0,k)T(a+1,0,k+1)
\\
&\qquad=\delta_k(q+1)T(\lfloor k/2 \rfloor+1,0,k+1)+q\sum_{a=0}^{\lfloor k/2\rfloor-1}T(a+1,0,k+1)
\\
&\qquad=\delta_k(q+1)T(\lfloor k/2 \rfloor+1,0,k+1)+q\sum_{a=1}^{\lfloor k/2\rfloor}T(a,0,k+1)
\\
&\qquad=\delta_k(q+1)T(\lfloor k/2 \rfloor+1,0,k+1)
\\
&\qquad\quad+\left( q\sum_{a=0}^{\lfloor (k+1)/2\rfloor}T(a,0,k+1) \right) -qT(0,0,k+1)
-\delta_kqT(\lfloor (k+1)/2\rfloor, 0, k+1)
\\
&\qquad=\delta_k(q+1)T(\lfloor k/2 \rfloor+1,0,k+1)
+qC_{k+1}-qT(0,0,k+1)-\delta_kqT(\lfloor (k+1)/2\rfloor, 0, k+1).
\end{align*}
And:
\begin{align*}
\sum_{a=0}^{\lfloor k/2\rfloor}n_4(a,0,k)T(a+1,1,k+1)
&=V^2\sum_{a=0}^{\lfloor k/2\rfloor}n_4(a,0,k)T(a,0,k-1)
\\
&=V^2\sum_{a=0}^{\lfloor (k-1)/2\rfloor}n_4(a,0,k)T(a,0,k-1)
\\
&=\delta_kq^2V^2T(\lfloor (k-1)/2\rfloor,0,k-1)+q^3V^2\sum_{a=0}^{\lfloor (k-1)/2\rfloor}T(a,0,k-1)
\\
&=\delta_kq^2V^2T(\lfloor (k-1)/2\rfloor,0,k-1)+q^3V^2C_{k-1}.
\end{align*}
And:
\begin{align*}
V\sum_{a=0}^{\lfloor k/2\rfloor}n_5(a,0,k)T(a,0,k)
&=n_5(0,0,k)VT(0,0,k)+V\sum_{a=1}^{\lfloor k/2\rfloor}n_5(a,0,k)T(a,0,k)
\\
&=(q-1)VT(0,0,k)+(q^2-1)V\sum_{a=1}^{\lfloor k/2\rfloor}T(a,0,k)
\\
&=(q-q^2)VT(0,0,k)+(q^2-1)V\sum_{a=0}^{\lfloor k/2\rfloor}T(a,0,k)
\\
&=(q-q^2)VT(0,0,k)+(q^2-1)VC_{k}.
\end{align*}

Adding, we have
\begin{align*}
XC_k& = (q+1)C_{k+1}+(q^2-1)VC_{k}+(q^2+q^3)V^2C_{k-1}+qT(0,1,k+1)+(q-q^2)Va_k-qa_{k+1}
\\
&\quad+\delta_k(q+1)T(\lfloor k/2 \rfloor+1,0,k+1)-\delta_k(q+1)T(\lfloor (k+1)/2\rfloor, 0, k+1)
\\
& = (q+1)C_{k+1}+(q^2-1)VC_{k}+(q^2+q^3)V^2C_{k-1}+qT(0,1,k+1)-qa_{k+1}+(q-q^2)Va_k.
\end{align*}
Hence,
\begin{align*}
X(C_k+\alpha(C_k))
&=(q+1)(C_{k+1}+\alpha(C_{k+1}))+(q^2-1)V(C_k+\alpha(C_k))+(q^2+q^3)V^2(C_{k-1}+\alpha(C_{k-1}))
\\
&\quad+qT^\alpha(0,1,k+1)-2qa_{k+1}+2(q-q^2)Va_k.
\end{align*}
By~\eqref{00krecursioneq134},
$$
qT^\alpha(0,1,k+1)=-a_{k+1}+(X-(q-1)V)a_k-q^3V^2a_{k-1}.
$$
Therefore,
\begin{align*}
X(C_k+\alpha(C_k))
&=(q+1)(C_{k+1}+\alpha(C_{k+1}))+(q^2-1)V(C_k+\alpha(C_k))+(q^2+q^3)V^2(C_{k-1}+\alpha(C_{k-1}))
\\
&\quad-a_{k+1}+(X-(q-1)V)a_k-q^3V^2a_{k-1}-2qa_{k+1}+2(q-q^2)Va_k\\
&=(q+1)(C_{k+1}+\alpha(C_{k+1}))+(q^2-1)V(C_k+\alpha(C_k))+(q^2+q^3)V^2(C_{k-1}+\alpha(C_{k-1}))
\\
&\quad-(2q+1)a_{k+1}+(X+(1+q-2q^2)V )a_k-q^3V^2a_{k-1}.
\end{align*}
Using $C_j+\alpha(C_j)=B_j+a_j$ for $j\in\Z_{\ge 0}$, we obtain
\begin{align*}
X(B_k+a_k)
&=(q+1)(B_{k+1}+a_{k+1})+(q^2-1)V(B_k+a_k)+(q^2+q^3)V^2(B_{k-1}+a_{k-1})
\\
&\quad-(2q+1)a_{k+1}+(X+(1+q-2q^2)V )a_k-q^3V^2a_{k-1}.
\end{align*}
This implies that
$$
 (q+1) B_{k+1} +((q^2-1) V-X) B_k + q^2(q+1) V^2 B_{k-1}
 = q a_{k+1} +(q^2-q) V a_k  -q^2 V^2 a_{k-1}.
$$
Dividing by~$q+1$, we have
\begin{align*}
& B_{k+1} +(q+1)^{-1}((q^2-1) V-X) B_k + q^2 V^2 B_{k-1}\\
&\qquad = (q+1)^{-1} q a_{k+1} +(q+1)^{-1}(q^2-q) V a_k  - (q+1)^{-1} q^2 V^2 a_{k-1}.
\end{align*}
This holds for~$k \geq 3$, and implies that~\eqref{Bkrecursioneq101} holds for~$k \geq 2$.
\end{proof}
We are now ready to show that the sequence ~$T(q^k)$ is recursive.

\begin{proposition}
\label{fullrecursion}
For~$k\in\Z_{\geq 0}$, let~$a_k$ and~$B_k$ be as in Lemma~\ref{00krecursion}
and Lemma~\ref{Bkrecursion}.
Define
\begin{equation}
\label{fullrecursioneq0p5}
\left\{
\begin{array}{l}
c_0=1,\\
c_1= (q^2-1) V-X,\\
c_2 = (q^3+q^2) V^2 + (q-q^2) V X + q Y,\\
c_3= (q^5-q^3) V^3 - q^3 V^2 X,\\
c_4=q^6 V^4.
\end{array}
\right\}
\end{equation}
Then
\begin{equation}
\label{fullrecursioneq1}
\sum_{i=0}^4c_iB_{k-i}=0
\qquad \text{for $k\ge 5$},
\end{equation}
and
\begin{equation}
\label{fullrecursioneq2}
\sum_{i=0}^4c_iT(q^{4-i})=\sum_{i=0}^4c_iT(q^{5-i}) = 0.
\end{equation}
\end{proposition}
\begin{proof}
We will first prove that
\begin{equation}
\label{fullrecursioneq2p5}
\sum_{i=0}^4c_ia_{k-i}=0\qquad \text{for~$k \geq 5$}.
\end{equation}
By~\eqref{00krecursioneq1}, if~$k \geq 5$, then
$$
\sum_{i=0}^4c_ia_{k-i+4}
=A\sum_{i=0}^4c_ia_{k-i+3}
+B\sum_{i=0}^4c_ia_{k-i+2}
+C\sum_{i=0}^4c_ia_{k-i+1}
+D\sum_{i=0}^4c_ia_{k-i}.
$$
Hence, to prove our claim~\eqref{fullrecursioneq2p5}, it suffices to prove that
\begin{equation}
\label{fullrecursioneq3}
\sum_{i=0}^4c_ia_{8-i}
=\sum_{i=0}^4c_ia_{7-i}
=\sum_{i=0}^4c_ia_{6-i}
=\sum_{i=0}^4c_ia_{5-i}=0.
\end{equation}
Using a computer algebra system,
and the multiplication tables
(Table~\ref{XHptable} and Table~\ref{YHptable})
one may directly verify~\eqref{fullrecursioneq3}; this proves~\eqref{fullrecursioneq2p5}.
To prove~\eqref{fullrecursioneq1}, first let~$k \geq 6$.
By~\eqref{Bkrecursioneq101}, we have:
\begin{align*}
&\sum_{i=0}^4 c_i B_{k-i+2}\\
&\qquad =  (q+1)^{-1}q\sum_{i=0}^4c_ia_{k-i+2}
+(q+1)^{-1}(q^2-q)V\sum_{i=0}^4c_ia_{k-i+1}
-(q+1)^{-1}q^2V^2\sum_{i=0}^4c_ia_{k-i}\\
&\qquad \quad -(q+1)^{-1}((q^2-1)V-X)\sum_{i=0}^4c_iB_{k-i+1}
-q^2 V^2\sum_{i=0}^4c_iB_{k-i}.
\end{align*}
By~\eqref{fullrecursioneq2p5}, the first three terms on the right hand side of the above expression are zero, and so we have
\begin{align*}
\sum_{i=0}^4c_iB_{k-i+2}=-(q+1)^{-1}((q^2-1)V-X)\sum_{i=0}^4c_iB_{k-i+1}
-q^2 V^2\sum_{i=0}^4c_iB_{k-i}.
\end{align*}
To prove~\eqref{fullrecursioneq1} it now suffices to show that
$$
\sum_{i=0}^4c_iB_{5-i}=\sum_{i=0}^4c_iB_{6-i}=0,
$$
which is verified using a computer algebra system. The
claim~\eqref{fullrecursioneq2} also follows from direct calculation using a computer algebra system.
\end{proof}

We now prove the main result of this section.

\begin{theorem}
\label{generalrationalitytheorem}
\label{generalrationality}
If~$k \geq 4$, then
\begin{equation}
\label{generalrationalitytheoremeq1}
\sum_{i=0}^4 c_i T(q^{k-i}) = 0
\end{equation}
where~$c_i$ for~$0 \leq i \leq 4$ are as in \eqref{fullrecursioneq0p5}.
\end{theorem}
\begin{proof}

We will  prove that~\eqref{generalrationalitytheoremeq1} holds for~$k \geq 4$
by induction on~$k$. The cases~$k=4$ and~$k=5$ are~\eqref{fullrecursioneq2}. Assume
that~$k \geq 6$ and that~\eqref{generalrationalitytheoremeq1} holds for~$k-1$.
We have:
\begin{align}
\sum_{i=0}^4 c_i T(q^{k-i})
& = \sum_{i=0}^4 c_i  B_{k-i} + \sum_{i=0}^4 c_i  V T(q^{k-i-1})
\qquad \text{(by~\eqref{generalrationalitytheoremeq5}; note that~$k-i \geq 2$)} \nonumber \\
& = \sum_{i=0}^4 c_i  B_{k-i} +  V \sum_{i=0}^4 c_i  T(q^{(k-1)-i}) \nonumber \\
& = 0 \qquad \text{(by~\eqref{fullrecursioneq1} and the induction hypothesis)}. \nonumber
\end{align}
It follows that~\eqref{generalrationalitytheoremeq1} holds for~$k \geq 4$.
\end{proof}

As mentioned at the beginning of this section, Theorem~\ref{generalrationalitytheorem}
implies that the  power series~$\sum_{k=0}^\infty T(q^k) t^k$ in the
indeterminate~$t$ is a rational function in~$t$.
In fact, by~\eqref{recideq},~\eqref{fullrecursioneq0p5}, and~\eqref{fullrecursioneq2p5},
we have
$$
\sum_{k=0}^\infty T(q^k) t^k=\frac{P(t)}{Q(t)},
$$
where
$$
P(t)=
1+ q^2 V t -q^2 V^2 t^2 - q^4 V^3 t^3 =
(1 - q  V t) (1 + q  V t) (1 + q^2  V t)
$$
and
\begin{align*}
Q(t) & = 1 +
 \left((-1 + q^2) V - X\right)t
+ q  \left((1 + q^2) V^2 + V (X - q X) + Y\right) t^2\\
&\quad +q^3  V^2 \left((-1 + q^2) V - X\right) t^3  + q^6  V^4 t^4.
\end{align*}
After some algebra, this gives the same results as in \cite{JLPR} for~$\sum_{k=0}^\infty \chi(T(q^k)) t^k$
in the case of characters~$\chi$ of type~$1$ and type~$2$.

\section{A quaternion algebra realization}
\label{sec5}
\label{quaternionalgsec}
We end the paper by exploring an interesting observation about the structure of the Hecke algebra ~$\mathcal H$.
Let~$\Z[2^{-1}]$ be the commutative subring of~$\Q$
generated by~$\Z$ and~$2^{-1}$. In this section we
prove that the graded ring~$\Z[2^{-1}] \otimes_\Z \mathcal{H}$
is isomorphic to a split graded quaternion algebra over~$\Z[2^{-1}] \otimes_\Z Z(\mathcal{H})$, where~$Z(\mathcal{H})$ is the center of~$\mathcal{H}$.
More precisely, we will define an abstract commutative graded polynomial ring~$R$ (which will be isomorphic to~$\Z[2^{-1}] \otimes_\Z \mathcal{H}$),
a graded quadratic~$R$-algebra~$S$ with involution~$\sigma \colon S \to S$ (which will be isomorphic to~$\Z[2^{-1}] \otimes_\Z \mathcal{H}^\alpha$),
and an element~$z$ in~$R$ such that~$\Z[2^{-1}] \otimes_\Z \mathcal{H}$
is isomorphic to the graded quaternion algebra~$A$ over~$R$ defined by~$\sigma$,~$z$,
and~$S$; we will prove that the quaternion algebra~$A$ is split in the sense that~$A$ admits
an embedding into the~$R$-algebra~$\Mat(2,R)$ of~$2 \times 2$ matrices
over~$R$. As a consequence, we obtain a concrete embedding
of~$\Z[2^{-1}] \otimes_\Z \mathcal{H}$ into~$\Mat(2,R)$: see
Corollary~\ref{Hquatcor}.

We begin by letting~$R$ be the commutative polynomial ring in the indeterminates~$z_1$,~$z_2$,
and~$z_3$ over the commutative ring~$\Z[2^{-1}]$. We give~$R$ a~$\Z$-grading
by letting~$R_n=0$ for~$n \in \Z$ with~$n<0$, and assigning~$z_1$ degree~$1$
and~$z_2$ and~$z_3$ degree~$2$.
Letting~$t$ be another indeterminate, define a commutative~$R$-algebra~$S$ by
$$
S = R[t]/\left(t^2 - z_3 \right).
$$
Let~$\mathbf{i}$ be the image of~$t$ in~$S$. Then
\begin{equation}
\label{Sdefeq}
S = R \oplus R\, \mathbf{i}
\end{equation}
as~$R$-modules,
and~$\mathbf{i}^2 = z_3$.  Define~$\sigma \colon S \to S$ by the
formula~$\sigma (a+b\,\mathbf{i}) = a-b\,\mathbf{i}$
for~$a,b \in R$. Then~$\sigma $ is an involution on~$S$.
Define~$\tr_S \colon R \to R$ and~$\nm_S \colon R \to R$
by~$\tr_S(x) = x+\sigma (x)$ and~$\nm_S (x) = x \sigma (x)$ for~$x \in R$.
Evidently, the elements of~$R$ are fixed by~$\sigma$
and~$\tr_S(x)$ and~$\nm_S(x)$ are in~$R$ for all~$s \in S$. In the language
of~\cite{Knus}, Chapter~I, Section~1,~$\sigma$ is a standard
involution and~$S$ is a quadratic~$R$-algebra. For~$n \in \Z$
we define
\begin{equation}
\label{Sgradingdefeq}
S_n = R_n \oplus R_{n-1} \mathbf{i}.
\end{equation}
Then~$S_m S_n \subset S_{m+n}$ for~$m,n \in \Z$, so
that~$S$ is a~$\Z$-graded~$R$-algebra. We note that~$\mathbf{i}$
has degree~$1$ and~$\sigma$
preserves this grading, i.e.,~$\sigma(S_n) \subset S_n$
for~$n \in \Z$.
For the next lemma, we recall from Theorem~5.5 of \cite{JLPR} that the center~$Z(\mathcal{H})$
is generated by the three elements
\begin{equation}
\label{Zedefeq}
Z_1 = X - (q^2-1) V, \qquad Z_2 = (q-1) VX - (Y_1+Y_2), \qquad Z_3 = V^2,
\end{equation}
and that~$Z_1$,~$Z_2$, and~$Z_3$ are algebraically independent over~$\Z$.
Also, it is evident that~$Z(\mathcal{H}) \subset \mathcal{H}^\alpha$,
so that~$\mathcal{H}^\alpha$ may be regarded as a~$Z(\mathcal{H})$-algebra.
By Proposition~\ref{comalglemma} and Lemma~\ref{Hcircpolylemma},
the elements~$V$,~$X$, and~$Y=Y_1 + Y_2$ are contained in~$\mathcal{H}^\alpha$,
generate~$\mathcal{H}^\alpha$, and are algebraically independent over~$\Z$.

\begin{lemma}
\label{jisolemma}
Define
$$
\iota \colon R
\longrightarrow
\Z \left[ 2^{-1} \right] \otimes_\Z Z (\mathcal{H})
$$
by defining
$$
\iota (z_1) = Z_1, \qquad \iota (z_2) = Z_2, \qquad \iota (z_3) = Z_3.
$$
Extend~$\iota$ to a function
\begin{equation}
\label{jisolemmaeq1}
\iota \colon S
\longrightarrow
\Z \left[ 2^{-1} \right] \otimes_\Z \mathcal{H}^\alpha
\end{equation}
by the formula~$\iota (a +b\, \mathbf{i}) = \iota (a) + \iota (b) V$ for~$a,b \in R$.
Then~$\iota$ is an isomorphism of graded rings
such that~$\iota(R) =  \Z \left[ 2^{-1} \right] \otimes Z (\mathcal{H})$.
\end{lemma}
\begin{proof}
It is straightforward to verify that~$\iota$ is homomorphism of graded rings
and that
\begin{align}
&\iota (\mathbf{i})  = V, \label{jisolemmaeq2} \\
&\iota ( z_1 +(q^2-1)\, \mathbf{i})  = X, \label{jisolemmaeq3} \\
&\iota (-z_2+ (q+1) (q-1)^2 z_3 + (q-1) z_1\, \mathbf{i} )   = Y. \label{jisolemmaeq4}
\end{align}
Define a ring homomorphism~$\kappa \colon \Z\left[ 2^{-1} \right] \otimes_\Z \mathcal{H}^\alpha
\longrightarrow S$
by setting
\begin{align}
\kappa (V) & =\mathbf{i},\label{jisolemmaeq5}  \\
\kappa (X) & = z_1 +(q^2-1)\, \mathbf{i},  \label{jisolemmaeq6} \\
\kappa (Y) &= -z_2 +(q-1)^2 (q+1) z_3 + (q-1) z_1\, \mathbf{i}. \label{jisolemmaeq7}
\end{align}
Using~\eqref{jisolemmaeq2},~\eqref{jisolemmaeq3},~\eqref{jisolemmaeq4},~\eqref{jisolemmaeq5},
~\eqref{jisolemmaeq6}, and~\eqref{jisolemmaeq7},
calculations show that~$\iota$ and~$\kappa$ are inverses of each other. Thus, the ring homomorphism~$\iota$
is an isomorphism.
\end{proof}

It follows from Lemma~\ref{jisolemma} that
there exists an involution on~$\Z \left[ 2^{-1} \right] \otimes \mathcal{H}^\alpha$, which we
also refer to as~$\sigma$, such that the following diagram commutes:
$$
\begin{CD}
\Z \left[ 2^{-1} \right] \otimes_\Z \mathcal{H}^\alpha @> \kappa >\sim > S \\
@V\sigma VV @VV \sigma V \\
Z \left[ 2^{-1} \right] \otimes_\Z \mathcal{H}^\alpha  @< \iota <\sim <  S
\end{CD}
$$
Calculations show that
\begin{align}
\sigma(V) & = -V, \label{sigforeq1} \\
\sigma(X) & = X - 2 (q^2-1) V,\label{sigforeq2} \\
\sigma (Y)  & = 2(q+1)(q-1)^2 V^2 - 2(q-1) VX +Y. \label{sigforeq3}
\end{align}

Next, define
\begin{align}
 \label{zdefineeq}
z& =   e^2  -f^2 z_3\in R
\end{align}
where~$e = z_2-\left(1 + q + q^2 + q^3 \right) z_3$ and~$f= (1+q) z_1$.
The element~$z$ has degree~$4$.
Let~$A$ be the quaternion algebra determined by~$S$,~$\sigma$,
and~$z$. Thus,~$A$ is a free~$S$-module of rank two, given as
\begin{equation}
\label{quaterniondefeq}
A = S \oplus S\, \mathbf{j},
\end{equation}
with multiplication determined by the conditions
\begin{equation}
\label{multruleseq}
\mathbf{j}^2 = z \quad \text{and} \quad \mathbf{j}\, s = \sigma(s)\, \mathbf{j}
\end{equation}
for~$s \in S$.
We have
$$
A = R \oplus R\, \mathbf{i} \oplus R\, \mathbf{j} \oplus R\, \mathbf{ij}.
$$
For~$n \in \Z$ define
\begin{equation}
\label{Andefeq}
A_n = S_n \oplus S_{n-2}\, \mathbf{j}.
\end{equation}
Then~$A_m A_n \subset A_{m+n}$ for~$m,n \in \Z$, and thus~$A$ is a~$\Z$-graded algebra
with the element~$\mathbf{j}$ having degree 2.
If~$n \in \Z$, then
\begin{equation}
\label{Angradingeq}
A_n = R_n \oplus R_{n-1} \, \mathbf{i} \oplus R_{n-2} \, \mathbf{j} \oplus R_{n-3} \,\mathbf{ij}.
\end{equation}
The standard involution~$\sigma$ on~$S$
extends to a standard involution, which we also call~$\sigma$, on the~$R$-algebra~$A$
via the formula~$\sigma( a+ b\, \mathbf{j}) = \sigma(a) - b\, \mathbf{j}$,
where~$a,b \in S$. If~$a,b,c,d \in R$, then
$$
\sigma( a + b\, \mathbf{i} + c\, \mathbf{j} + d\, \mathbf{ij} )
= a - b\, \mathbf{i} - c\, \mathbf{j} -d\, \mathbf{ij}.
$$
The standard involution~$\sigma$ preserves our grading on~$A$.
For later use, we note that the equation~\eqref{zdefineeq} implies that there is an embedding
of~$R$-algebras
\begin{equation}
\label{Aembedeq}
\nu \colon A \hookrightarrow \Mat(2,R)
\end{equation}
that is determined by the assignments
\begin{equation}
\label{ijmapeq}
\nu(\mathbf{i}) = \begin{bmatrix} & 1 \\ z_3 & \end{bmatrix}
\quad \text{and} \quad
\nu(\mathbf{j}) =
\begin{bmatrix}
e & f \\ - f z_3 & - e
\end{bmatrix}.
\end{equation}
\begin{lemma}
\label{qasupportlemma}
There is a direct sum decomposition
\begin{equation}
\label{qasupportlemmaeq3}
\Z \left[ 2^{-1} \right] \otimes_\Z \mathcal{H}
=
\left( \Z \left[ 2^{-1} \right] \otimes_\Z \mathcal{H}^\alpha  \right) \oplus
\left( \Z \left[ 2^{-1} \right] \otimes_\Z \mathcal{H}^\alpha \right) \left( Y_1-Y_2 \right).
\end{equation}
Define
\begin{equation}
\label{qasupportlemmaeq1}
Z=   \left( Z_2-\left(1 + q + q^2 + q^3 \right) Z_3\right)^2  -\left((1+q) Z_1\right)^2 Z_3
\in Z(\mathcal{H}).
\end{equation}
Then~$(Y_1-Y_2)^2 = Z$, and
\begin{equation}
\label{qasupportlemmaeq2}
\left( Y_1 -Y_2\right) U = \sigma (U) \left( Y_1 -Y_2\right)
\end{equation}
for~$U \in \Z \left[ 2^{-1} \right] \otimes_\Z  \mathcal{H}^\alpha$.
\end{lemma}
\begin{proof}
Using the definitions of~$Z_1$,~$Z_2$, and~$Z_3$ from~\eqref{Zedefeq}, we have
\begin{align*}
Z& =  \left( Z_2-\left(1 + q + q^2 + q^3 \right) Z_3\right)^2  -\left((1+q) Z_1\right)^2 Z_3 \\
& = \left( (q-1) VX -Y -\left(1 + q + q^2 + q^3 \right) V^2\right)^2  -\left((1+q) (X-(q^2-1) V)\right)^2 V^2 \\
& =4q^2 (q+1)^2 V^4 + 4q (q^2-1) V^3 X -
 4 q V^2 X^2 + 2(1 +  q +  q^2 +  q^3) V^2 Y + 2(1 -  q) V X Y + Y^2.
\end{align*}
We also have
\begin{align*}
&\left( Y_1 - Y_2 \right)^2\\
&\qquad  = Y_1^2 -Y_1Y_2-Y_2Y_1+Y_2^2 \\
&\qquad  = \left( Y_1+Y_2 \right)^2 -2 Y_1 Y_2 - 2 Y_2 Y_1 \\
&\qquad  =  \left( Y_1+Y_2 \right)^2 \\
& \qquad \quad + 2 \left( q^2 (q+1)^2 V^4 + q(q^2-1) V^3 X - q V^2 X^2 - (q-1)VXY_2
+  q(q+1) V^2 Y_1 +  (1+q^3) V^2 Y_2 \right)\\
& \qquad \quad + 2 \left( q^2 (q+1)^2 V^4 + q(q^2-1) V^3 X - q V^2 X^2 - (q-1)VXY_1
+  q(q+1) V^2 Y_2 +  (1+q^3) V^2 Y_1 \right)\\
& \qquad \quad \text{(by~(78) of~\cite{JLPR}; also,~$\alpha(Y_1) = Y_2$ and~$\alpha(Y_2) = Y_1$
by~\eqref{alphageneq})} \\
&\qquad  =  \left( Y_1+Y_2 \right)^2 +  4 q^2 (q+1)^2 V^4 +4 q(q^2-1) V^3 X - 4q V^2 X^2\\
& \qquad \quad + 2 \left(1+q+q^2+q^3 \right) V^2 (Y_1+Y_2) -2(q-1) VX (Y_1+Y_2)  \\
&\qquad  =  Y^2 +  4 q^2 (q+1)^2 V^4 +4 q(q^2-1) V^3 X - 4q V^2 X^2
+ 2 \left(1+q+q^2+q^3 \right) V^2 Y +2(1-q) VX Y.
\end{align*}
Comparing, we conclude that~$(Y_1-Y_2)^2=Z$.
By Lemma~\ref{Hprimemodlemma} we
have~$\mathcal{H} = \mathcal{H}^\alpha + \mathcal{H}^\alpha Y_1$.
Since~$Y_1$ is~$2^{-1} \left( Y_1+Y_2 \right) + 2^{-1} \left( Y_1 - Y_2 \right)$
and~$Y_1 + Y_2 \in \mathcal{H}^\alpha$, this implies that
\begin{equation}
\label{qasupportlemmaeq4}
\Z \left[ 2^{-1} \right] \otimes_\Z \mathcal{H}
=
\left( \Z \left[ 2^{-1} \right] \otimes_\Z \mathcal{H}^\alpha  \right) +
\left( \Z \left[ 2^{-1} \right] \otimes_\Z \mathcal{H}^\alpha \right) \left( Y_1-Y_2 \right).
\end{equation}
To see that the sum in~\eqref{qasupportlemmaeq4} is direct, let~$A,B \in \mathcal{H}^\alpha$,
and assume that~$A = B \left(Y_1-Y_2\right)$. Applying the automorphism~$\alpha$ of~$\mathcal{H}$,
we obtain~$A = B \alpha \left(Y_1-Y_2\right) = - B \left(Y_1 -Y_2 \right)$.
It follows that~$2A=0$, so that~$A$ is zero, and the sum is direct.
To prove that~\eqref{qasupportlemmaeq2} holds for all~$U \in \mathcal{H}^\alpha$
it will suffice to verify~~\eqref{qasupportlemmaeq2}
for the generators~$V$,~$X$, and~$Y=Y_1+Y_2$ of~$\mathcal{H}^\alpha$.
By~(73) and~(74) of~\cite{JLPR},
$$
(Y_1-Y_2) V
 = Y_1 V - Y_2 V
 = V Y_2 - V Y_1
 = \sigma(V) (Y_1-Y_2).
$$
And
\begin{align}
(Y_1-Y_2) X
& = Y_1 X - Y_2 X \nonumber \\
& = XY_1 +(1-q^2) VY_1 -(1-q^2) VY_2
-XY_2 +(1-q^2) VY_1 - (1-q^2) VY_2 \nonumber \\
& \quad \text{(by~(75) and~(76) of~\cite{JLPR})} \nonumber \\
& = X(Y_1-Y_2) -2(q^2-1) V(Y_1 -Y_2)\nonumber \\
& = \sigma(X)(Y_1-Y_2)\qquad \text{(by~\eqref{sigforeq2})}.\nonumber
\end{align}
Finally,
\begin{align}
(Y_1-Y_2) Y
& = (Y_1-Y_2) (Y_1+Y_2) \nonumber  \\
& = Y_1^2 -Y_2 Y_1 +Y_1Y_2 -Y_2^2 \nonumber \\
& = (q-1)^2(q+1) V^2 Y_1 - (q-1)^2 (q+1) V^2 Y_2 - (q-1) VXY_1 +(q-1) V XY_2 \nonumber \\
&\quad +Y_1^2 - Y_2^2 \qquad \text{(by~(77) of~\cite{JLPR})}, \nonumber
\end{align}
and
\begin{align}
\sigma(Y) (Y_1-Y_2)
& = \left( 2 (q+1) (q-1)^2 V^2 -2 (q-1) V X + Y_1+Y_2 \right)(Y_1-Y_2)
\qquad \text{(by~\eqref{sigforeq3})} \nonumber \\
& = 2 (q+1) (q-1)^2 V^2 Y_1 -2 (q-1) V XY_1 + Y_1^2+Y_2Y_1 \nonumber \\
& \quad -2 (q+1) (q-1)^2 V^2Y_2 +2 (q-1) V X Y_2- Y_1Y_2-Y_2^2 \nonumber\\
& = Y_2Y_1 - Y_1Y_2 + 2 (q+1) (q-1)^2 V^2 Y_1 -2 (q-1) V XY_1  \nonumber \\
& \quad -2 (q+1) (q-1)^2 V^2Y_2 +2 (q-1) V X Y_2 + Y_1^2-Y_2^2 \nonumber \\
& = -(q-1)^2(q+1) V^2 Y_1 +(q-1)^2 (q+1) V^2 Y_2 + (q-1) VXY_1 -(q-1) VXY_2\nonumber \\
& \quad + 2 (q+1) (q-1)^2 V^2 Y_1 -2 (q-1) V XY_1  \nonumber \\
& \quad -2 (q+1) (q-1)^2 V^2Y_2 +2 (q-1) V X Y_2 + Y_1^2-Y_2^2 \qquad \text{(by~(77) of~\cite{JLPR})} \nonumber \\
& = (q-1)^2(q+1) V^2 Y_1 -(q-1)^2 (q+1) V^2 Y_2 - (q-1) VXY_1 +(q-1) VXY_2\nonumber \\
& \quad   + Y_1^2-Y_2^2, \nonumber
\end{align}
so that~$(Y_1-Y_2) Y = \sigma(Y) (Y_1-Y_2)$.
\end{proof}

\begin{theorem}
\label{Hquaterniontheorem}
Extend the isomorphism~$\iota \colon S \stackrel{\sim}{\longrightarrow} \Z[2^{-1}] \otimes_\Z \mathcal{H}^\alpha$ to a function
\begin{equation}
\label{Hquaterniontheoremeq1}
\iota \colon A \longrightarrow \Z [2^{-1}] \otimes_\Z \mathcal{H}
\end{equation}
by defining
\begin{equation}
\label{Hquaterniontheoremeq1p5}
\iota (a + b\, \mathbf{j} ) = \iota (a) + \iota(b) \left(Y_1 - Y_2 \right)
\end{equation}
for~$a,b \in S$. Then~$\iota$ is an isomorphism of graded rings.
\end{theorem}
\begin{proof}
The function~$\iota$ is clearly additive. A calculation
shows that~$\iota$ is multiplicative; for this, it is useful
to note that~$\iota(\sigma (s)) = \sigma (\iota (s)))$
for~$s \in S$ and also that~$\iota(z) = Z$. Thus,~$\iota$
is a ring homomorphism.
To prove that~$\iota$ is injective, assume that~$a,b \in S$
are such that~$\iota(a+b \, \mathbf{j})=0$.
Then~$ \iota(a) + \iota(b) (Y_1-Y_2) =0$,
and by~\eqref{qasupportlemmaeq3} we obtain~$\iota(a) =0$
and~$\iota(b) (Y_1-Y_2) =0$.
Since~$\iota|_S$ is injective by Lemma~\ref{jisolemma}
we have~$a=0$.
Multiplying~$\iota(b) (Y_1-Y_2) =0$ by~$Y_1-Y_2$
implies that~$\iota(b) Z =0$ since~$(Y_1-Y_2)^2=Z$.
The elements~$\iota(b)$ and~$Z$ are in~$\Z [2^{-1}] \otimes_\Z \mathcal{H}^\alpha$;
since this ring is an integral domain by Lemma~\ref{comalglemma}
and Lemma~\ref{Hcircpolylemma}, and since~$Z \neq 0$, we obtain~$b=0$.
The direct sum decomposition~\eqref{qasupportlemmaeq3}
also implies that~$\iota$ is surjective. To see that~$\iota$
preserves gradings, let~$n \in \Z$. Then by~\eqref{Angradingeq},
$$
\iota (A_n) =
\iota (R_n) + \iota (R_{n-1}) V + \iota (R_{n-2}) (Y_1-Y_2) +\iota (R_{n-3} ) V (Y_1-Y_2)
\subset
\Z \left[ 2^{-1} \right] \otimes_\Z \mathcal{H}_n
$$
since~$\iota|_R$ is a homomorphism of graded rings,~$V$
has degree~$1$, and~$Y_1-Y_2$ has degree~$2$.
It follows that~$\iota$ is an isomorphism of graded rings.
\end{proof}

To end the section, we present the following corollary, which details the embedding of~$\Z[2^{-1}] \otimes_\Z \mathcal{H}$ into~$\Mat(2,R)$.

\begin{corollary}
\label{Hquatcor}
For~$n \in \Z$ define
\begin{equation}
\label{Hquatcoreq1}
\Mat(2,R)_n = \begin{bmatrix} R_n & R_{n-1} \\ R_{n+1} & R_n \end{bmatrix}.
\end{equation}
With this definition,~$\Mat(2,R)$ is a $\Z$-graded ring. Let~$c$
be the composition
\begin{equation}
\label{Hquatcoreq2}
\Z \left[ 2^{-1} \right] \otimes_\Z \mathcal{H}
\stackrel{\iota^{-1}}{\longrightarrow} A \stackrel{\nu}{\hookrightarrow} \Mat(2,R),
\end{equation}
where~$\nu \colon A \hookrightarrow \Mat(2,R)$ is as in~\eqref{Aembedeq}.
The function~$c$ is an injective homomorphism of graded
rings such that
\begin{equation}
\label{Hquatcoreq3}
c(V) = \begin{bmatrix} &1 \\ z_3 & \end{bmatrix}, \qquad
c(X) =  \begin{bmatrix} z_1 & q^2 -1 \\ (q^2-1) z_3 & z_1 \end{bmatrix}
\end{equation}
and
\begin{equation}
\label{Hquatcoreq4}
c(Y_1) =
\begin{bmatrix}
-q (q+1) z_3 & q z_1 \\ -z_1 z_3 & -z_2 +(q^3+1) z_3
\end{bmatrix},
\qquad
c(Y_2) =
\begin{bmatrix}
 -z_2 +(q^3+1) z_3 & -z_1 \\
q z_1 z_3 & -q (q+1) z_3
\end{bmatrix}.
\end{equation}
\end{corollary}
\begin{proof}
A straightforward calculation shows that~\eqref{Hquatcoreq1}
defines a grading for~$\Mat(2,R)$. The function~$c$
is an injective ring homomorphism because it is the composition
of injective ring homomorphisms.
We have
\begin{align}
c(V) &= \nu \left( \iota^{-1} (V) \right) \nonumber \\
&= \nu (\mathbf{i} ) \qquad \text{(by~\eqref{jisolemmaeq2})} \nonumber \\
& = \begin{bmatrix} & 1 \\ z_3 & \end{bmatrix} \text{(by~\eqref{ijmapeq})},\nonumber
\end{align}
and
\begin{align}
c(X) & = \nu \left( \iota^{-1} (X) \right) \nonumber \\
& = \nu (  z_1 +(q^2-1)\, \mathbf{i} ) \qquad\text{(by \eqref{jisolemmaeq3})} \nonumber \\
& = \begin{bmatrix} z_1 & \\ & z_1 \end{bmatrix}
+ \begin{bmatrix} & q^2-1 \\ (q^2-1) z_3 & \end{bmatrix}
\text{(by~\eqref{ijmapeq})}\nonumber \\
& = \begin{bmatrix} z_1 &  q^2-1 \\  (q^2-1) z_3  & z_1 \end{bmatrix}. \nonumber
\end{align}
This proves the formulas for~$c(V)$ and~$c(X)$ as in~\eqref{Hquatcoreq3}.
Next,
\begin{align}
c(Y_1-Y_2)
& = \nu \left( \iota^{-1} ( Y_1 - Y_2) \right) \nonumber  \\
& =  \nu ( \mathbf{j} ) \qquad \text{(by~\eqref{Hquaterniontheoremeq1p5})}\nonumber \\
& =
\begin{bmatrix}
z_2 -(1+q+q^2+q^3) z_3 & (1+q) z_1 \\
- (1+q) z_1 z_3 & - \left( z_2 -(1+q+q^2+q^3) z_3 \right)
\end{bmatrix} \qquad \text{(by~\eqref{ijmapeq})}, \label{Hquatcoreq5}
\end{align}
and
\begin{align}
c(Y_1+Y_2)
& = \nu \left( \iota^{-1} (Y) \right) \nonumber \\
& = \nu \left( -z_2+ (q+1) (q-1)^2 z_3 + (q-1) z_1 \, \mathbf{i} \right)
\qquad \text{(by~\eqref{jisolemmaeq4})} \nonumber \\
& =
\begin{bmatrix}
-z_2+ (q+1) (q-1)^2 z_3 & \\
& -z_2+ (q+1) (q-1)^2 z_3
\end{bmatrix}
+
\begin{bmatrix}
(q-1) z_1  & \\
& (q-1) z_1
\end{bmatrix} \nu (\mathbf{i}) \nonumber \\
& =
\begin{bmatrix}
-z_2+ (q+1) (q-1)^2 z_3 & (q-1) z_1 \\
(q-1) z_1 z_3 &-z_2+ (q+1) (q-1)^2 z_3
\end{bmatrix}
\qquad \text{(by~\eqref{ijmapeq})}. \label{Hquatcoreq6}
\end{align}
The formulas for~$c(Y_1)$ and~$c(Y_2)$ in~\eqref{Hquatcoreq4}
follow from~\eqref{Hquatcoreq5} and~\eqref{Hquatcoreq6}.
Finally, that~$c$ respects the gradings
of the rings~$\Z \left[ 2^{-1} \right] \otimes_\Z \mathcal{H}$
and~$\Mat(2,R)$ follows from~\eqref{Hquatcoreq3} and~\eqref{Hquatcoreq4}.
\end{proof}

\bibliographystyle{plain}
\bibliography{PHA2}

\end{document}